\documentclass[11pt,english]{article}

\usepackage[T1]{fontenc}
\usepackage[latin9]{inputenc}
\usepackage{geometry}
\usepackage{babel}
\usepackage{amsmath}
\usepackage{amsthm}
\usepackage{amssymb}
\usepackage{setspace}
\makeatletter
\date{}

\usepackage{appendix}

\usepackage{bm}

\usepackage[nameinlink,capitalise,noabbrev]{cleveref}
\crefname{appsec}{Appendix}{Appendices}
\creflabelformat{enumi}{#2textup{#1}#3}

\begingroup
    \makeatletter
    \@for\theoremstyle:=definition,remark,plain\do{%
        \expandafter\g@addto@macro\csname th@\theoremstyle\endcsname{%
            \addtolength\thm@preskip\parskip
            }%
        }
\endgroup

\theoremstyle{plain}
\newtheorem{thm}{\protect\theoremname}
\crefname{thm}{Theorem}{Theorems}
\theoremstyle{definition}
\newtheorem{defn}[thm]{\protect\definitionname}
\theoremstyle{definition}
\newtheorem{claim}[thm]{\protect\claimname}
\theoremstyle{definition}

\theoremstyle{plain}

\crefname{prop}{Proposition}{Propositions}
\theoremstyle{plain}

\theoremstyle{plain}
\newtheorem{lem}[thm]{\protect\lemmaname}
\crefname{lem}{Lemma}{Lemmas}
\theoremstyle{plain}

\theoremstyle{definition}

\crefformat{equation}{#2(#1)#3}

\let\originalleft\left
\let\originalright\right
\renewcommand{\left}{\mathopen{}\mathclose\bgroup\originalleft}
\renewcommand{\right}{\aftergroup\egroup\originalright}
\usepackage{pgfplots}
\usetikzlibrary{pgfplots.groupplots}
\usepackage{verbatim}

\makeatletter
\renewcommand*{\UrlTildeSpecial}{%
  \do\~{%
    \mbox{%
      \fontfamily{ptm}\selectfont
      \textasciitilde
    }%
  }%
}%
\let\Url@force@Tilde\UrlTildeSpecial
\makeatother

\usepackage{tikz}
\tikzstyle{vertex}=[circle,draw=black,fill=black,inner sep=0,minimum size=0.2cm,text=white,font=\footnotesize]
\tikzset{every loop/.style={min distance=50,in=50,out=130,looseness=7}}

\usepackage[labelfont=bf,labelsep=period]{caption}

\usepackage{enumitem}

\makeatother
\providecommand{\claimname}{Claim}
\providecommand{\conjecturename}{Conjecture}
\providecommand{\corollaryname}{Corollary}
\providecommand{\definitionname}{Definition}
\providecommand{\examplename}{Example}

\providecommand{\lemmaname}{Lemma}
\providecommand{\propositionname}{Proposition}
\providecommand{\remarkname}{Remark}
\providecommand{\theoremname}{Theorem}

\begin{document}

\title{Non-trivially intersecting multi-part families}

\author{Matthew Kwan
\thanks{Department of Mathematics, ETH, 8092 Zurich. Email: \href{mailto:matthew.kwan@math.ethz.ch} {\nolinkurl{matthew.kwan@math.ethz.ch}}.}
\and Benny Sudakov
\thanks{Department of Mathematics, ETH, 8092 Zurich. Email: \href{mailto:benjamin.sudakov@math.ethz.ch} {\nolinkurl{benjamin.sudakov@math.ethz.ch}}.}
\and Pedro Vieira
\thanks{Department of Mathematics, ETH, 8092 Zurich. Email: \href{mailto:pedro.vieira@math.ethz.ch} {\nolinkurl{pedro.vieira@math.ethz.ch}}.}
}

\maketitle
\global\long\def\RR{\mathbb{R}}
\global\long\def\NN{\mathbb{N}}
\global\long\def\ZZ{\mathbb{Z}}
\global\long\def\range#1{\left[#1\right]}
\global\long\def\squplus{\sqcup}
\global\long\def\F{\mathcal{F}}
\global\long\def\A{\mathcal{A}}
\global\long\def\B{\mathcal{B}}
\global\long\def\G{\mathcal{G}}
\global\long\def\C{\mathcal{C}}
\global\long\def\cH{\mathcal{H}}
\global\long\def\HM{\mathrm{HM}}
\global\long\def\EKR{\mathrm{EKR}}
\global\long\def\alt{\mathrm{alt}}
\global\long\def\ll{{\boldsymbol\ell}}
\global\long\def\kunif{\prod_{s}\!\binom{\range{n_s}}{k_s}}

\begin{abstract}
We say a family of sets is \emph{intersecting} if any two of its sets intersect, and we say it is \emph{trivially intersecting} if there is an element which appears in every set of the family. In this paper we study the maximum size of a non-trivially intersecting family in a natural ``multi-part'' setting. Here the ground set is divided into parts, and one considers families of sets whose intersection with each part is of a prescribed size. Our work is motivated by classical results in the single-part setting due to Erd\H os, Ko and Rado, and Hilton and Milner, and by a theorem of Frankl concerning intersecting families in this multi-part setting. In the case where the part sizes are sufficiently large we determine the maximum size of a non-trivially intersecting multi-part family, disproving a conjecture of Alon and Katona. 

\end{abstract}

\section{Introduction}

We say that a family of sets $\F$ is \emph{intersecting} if the intersection of any two of its sets
is non-empty. Moreover,  we say that $\F$ is \emph{trivially
intersecting} if there is an element $i$ such that $i\in F$ for each set $F\in \F$. The Erd\H{o}s-Ko-Rado theorem~\cite{EKR61} says that if $\F\subseteq\binom{\range n}{k}$
is an intersecting family of $k$-element subsets of an $n$-element ground set, and if $1\le k\le n/2$,
then
\[
\left|\F\right|\le\binom{n-1}{k-1}.
\]
This bound is sharp: it is attained, for example, by the trivially intersecting family $\F^{\EKR}\left(n,k\right)$
consisting of all $k$-element subsets of $[n]$ which contain 1. In fact, for $k<n/2$ this is essentially the only extremal family. We remark that if $k>n/2$ then $\binom{\range n}{k}$ itself is intersecting.

The Erd\H os-Ko-Rado theorem is of fundamental importance in extremal
set theory, and many related questions have been asked and answered.
In particular, Hilton and Milner~\cite{HM67}
proved a stability version of the Erd\H os-Ko-Rado theorem, showing that for $2\le k< n/2$,
the maximum size of a non-trivially intersecting family $\F\subseteq\binom{\range n}{k}$
is
\[
M^{\HM}\left(n,k\right):=\binom{n-1}{k-1}-\binom{n-k-1}{k-1}+1.
\]
This bound is sharp: it is attained, for example, by the family $\F^{\HM}\left(n,k\right)$
consisting of the set $F=\left\{ 2,\dots,k+1\right\} $, in addition to all
possible sets that contain 1 and intersect $F$. Note that this family is
significantly smaller than the Erd\H os-Ko-Rado bound. In particular, for constant $k$ and $n\to\infty$,
we have $\left|\F^{\HM}\left(n,k\right)\right|=o\left(\left|\F^{\EKR}\left(n,k\right)\right|\right)$. We remark that if $k=1$ then every intersecting family is trivially intersecting.

\subsection{Multi-part intersecting families}

A natural ``multi-part'' extension of the Erd\H{o}s-Ko-Rado problem was introduced by Frankl~\cite{Fra96}, in connection with a result of Sali~\cite{Sal92} (see also~\cite{FT93}). For $p\ge1$ and
$n_{1},\dots,n_{p}\ge1$, our ground set is $\range{\sum_{s}n_{s}}=\{1,2,\dots,\sum_{s}n_{s}\}$. We interpret this ground set as the disjoint union of $p$ parts $\range{n_{1}},\dots,\range{n_p}$ and we write $\range{\sum_{s}n_{s}}=\bigsqcup_{s}\range{n_{s}}$. More generally, for sets $F_{1}\in2^{\range{n_{1}}},\dots,F_{p}\in2^{\range{n_{p}}}$ let $\bigsqcup_{s}F_{s}$ be the subset of $\bigsqcup_{s}\range{n_{s}}$ with $F_{s}$ in
part $s$,
 and for families $\F_{1}\subseteq 2^{\range{n_{1}}},\dots,\F_{p}\subseteq 2^{\range{n_{p}}}$ let $\prod_{s}\F_s=\{\bigsqcup_{s} F_s:F_s\in\F_s\}$. Consider $k_1\in\range{n_1},\dots k_p\in\range{n_p}$, so that $
\kunif$ is the collection of all subsets of $\bigsqcup_{s}\range{n_{s}}$ which have exactly $k_s$ elements in each part $s$. Families of the form $\F\subseteq \kunif$ are the natural generalization of $k$-uniform families to the multi-part setting. Note that a multi-part family is intersecting if any two of its sets intersect in at least one of the parts.

Frankl proved that for any $p\ge1$, any $n_1,\dots,n_p$ and any $k_1,\dots,k_p$ satisfying $1\le k_s\le n_s/2$, the maximum size of a multi-part intersecting family $\F\subseteq \kunif$ is
\[
\max_{t\in\range p}\binom{n_{t}-1}{k_{t}-1}\prod_{s\ne t}\binom{n_{s}}{k_{s}}=\left(\max_{t\in \range p}\frac{k_{t}}{n_{t}}\right)\prod_{s=1}^{p}\binom{n_{s}}{k_{s}}.
\]
This bound is sharp: it is attained, for example, by a product family of the form
\[
\binom{\range{n_{1}} }{ k_{1}}\times\dots\times\binom{\range{n_{t-1}} }{ k_{t-1}}\times\F^{\EKR}\left(n_{t},k_{t}\right)\times\binom{\range{n_{t+1}} }{ k_{t+1}}\times\dots\times\binom{\range{n_{p}}}{ k_{p}}.
\]

We remark that Frankl's theorem can be interpreted as a result about the size of the largest independent set in a certain product graph. 
Recall that the \emph{Kneser graph} $\mathrm{KG}_{n,k}$ is the graph
on the vertex set $\binom{[n]}{k}$, with an edge between each
pair of disjoint sets. An intersecting subfamily of $\binom{[n]}{k}$
corresponds to an independent set in $\mathrm{KG}_{n,k}$,
and an intersecting subfamily of $\kunif$
corresponds to an independent set in the graph (tensor) product
$\mathrm{KG}_{n_{1},k_{1}}\times\dots\times\mathrm{KG}_{n_{p},k_{p}}$. Therefore Frankl's theorem
is an immediate consequence of the general fact that $\alpha\left(G\times H\right)=\max\left\{ \alpha\left(G\right)\left|H\right|,\left|G\right|\alpha\left(H\right)\right\} $
for vertex-transitive graphs $G,H$. This fact was conjectured
by Tardif~\cite{Tar98} and recently proved by Zhang~\cite{Zha12}. 
The study of independent sets in graph products, particularly graph powers, has a long history, and there are many 
research papers devoted to this topic. In particular we mention the work of Alon, Dinur, Friedgut and Sudakov~\cite{ADFS04} characterizing maximum and near-maximum independent sets in powers of certain graphs, and the subsequent work of Dinur, Friedgut and Regev~\cite{DFR08} approximately characterizing \emph{all} independent sets in powers of a much wider range of graphs.

Alon and Katona~\cite{AK16} independently  rediscovered Frankl's
multi-part variant of the Erd\H os-Ko-Rado problem, and proved the
same result. Furthermore, they also asked for the maximum size of a non-trivially
intersecting family in this setting, and made the natural conjecture that for $p=2$, the maximum possible size is
\[
\max\left\{M^\HM(n_1,k_1)\binom{n_2}{k_2},\;\binom{n_1}{k_1}M^\HM(n_2,k_2)\right\},
\]
attained by one of the ``product'' families
\[
\F^{\HM}\left(n_{1},k_{1}\right)\times\binom{\range{n_2}}{k_2},\quad\binom{\range{n_1}}{k_1}\times\F^{\HM}\left(n_{2},k_{2}\right).
\]
This conjecture also appeared in a recent paper of Katona~\cite{Kat17}, which generalized Frankl's theorem in a different direction.

Somewhat surprisingly, this natural guess is not true in general. Consider
the family that contains $F=\left\{ 2,\dots,k_{1}+1\right\} \sqcup\range{k_{2}}$
in addition to every set $A=A_1\sqcup A_2\in\binom{\range{n_{1}}}{k_{1}}\times\binom{\range{n_{2}}}{k_{2}}$
such that $A_1$ contains 1 and $A$ intersects
$F$. Unless $k_{1}=k_{2}=1$ (in which case there is no non-trivially
intersecting family), this family is non-trivially intersecting and
has size
\[
M^{\text{alt}}(n_1,n_2,k_1,k_2):=\binom{n_{1}-1}{ k_{1}-1}\binom{n_{2}}{ k_{2}}-\binom{n_{1}-k_{1}-1}{k_{1}-1}\binom{n_{2}-k_{2}}{k_{2}}+1.
\]
Observe that this can be larger than Alon and Katona's conjecture (consider for example the case when $n_{1}=n_{2}=5$ and $k_{1}=k_{2}=2$).

If $n_1,n_2$ are large, we are able to prove that either one of the product constructions above, or this additional construction (or the corresponding construction where the roles of the parts are swapped) is best possible. That is,
\begin{align*}
\left|\F\right| & \le \max\left\{
M^\HM(n_1,k_1)\binom{n_2}{k_2},\;\binom{n_1}{k_1}M^\HM(n_2,k_2),\;M^{\text{alt}}(n_1,n_2,k_1,k_2),\;M^{\text{alt}}(n_2,n_1,k_2,k_1)
\right\}.
\end{align*}
This will be an immediate corollary of a more general theorem giving the maximum size of a non-trivially intersecting multi-part family for any number of parts. To state this theorem we define a variety of potentially extremal families. Consider any $p\ge1$ and any $k_1,\dots,k_p$ satisfying $1\le k_s\le n_s/2$. For any $t\in\range p$ and $S\subseteq [p]\setminus \{t\}$, define the family $\F_{t,S}^{\HM}=\F_{t,S}^{\HM}\left(n_{1},\dots,n_{p},k_{1},\dots,k_{p}\right)$ as follows. If $k_t>1$ then $\F_{t,S}^{\HM}$ consists of all sets $\bigsqcup_sF_s\in\kunif$ such that
one of the following conditions is satisfied:
\begin{itemize}
\item $F_t=\left\{ 2,3,\dots,k_{t}+1\right\} $ and $F_s=\range{k_{s}}$
for all $s\in S$, or,
\item $1\in F_t$, and either $F_t$ intersects $\left\{ 2,3,\dots,k_{t}+1\right\} $
or $F_s$ intersects $\range{k_{s}}$ for some $s\in S$.
\end{itemize}
If instead $k_{t}=1$ then $\F_{t,S}^{\HM}$ consists of every set $\bigsqcup_sF_s\in\kunif$ such that
one of the following conditions is satisfied:
\begin{itemize}
\item $F_s=\range{k_{s}}$
for all $s\in S$, or,
\item $F_t=\{1\}$ and $F_s$ intersects $\range{k_{s}}$ for some $s\in S$.
\end{itemize}
Unless $S=\emptyset$ and $k_t=1$, or $S=\{r\}$ and $k_t=k_r=1$ (for some $r$), the family $\F_{t,S}^{\HM}$ is non-trivially intersecting. In the case $k_t>1$ it has size
\begin{align*}
 M_{t,S}^{\HM}=&M_{t,S}^{\HM}\left(n_{1},\dots,n_{p},k_{1},\dots,k_{p}\right)\\
 :=&\left(\binom{n_{t}-1}{k_{t}-1}\prod_{s\in S}\binom{n_{s}}{k_{s}}-\binom{n_{t}-k_{t}-1}{k_{t}-1}\prod_{s\in S}\binom{n_{s}-k_{s}}{k_{s}}+1\right)\prod_{s\notin S\cup\left\{ t\right\} }\binom{n_{s}}{k_{s}},
\end{align*}
and in the case $k_t=1$, it has size
\begin{align*}
 & M_{t,S}^{\HM}=\left(\prod_{s\in S}\binom{n_{s}}{k_{s}}-\prod_{s\in S}\binom{n_{s}-k_s}{k_{s}}+\left(n_{t}-1\right)\right)\prod_{s\notin S\cup\left\{ t\right\} }\binom{n_{s}}{k_{s}}.
\end{align*}

Our main theorem is as follows, showing that one of the families $\F^\HM_{t,S}$ is extremal if the $n_s$ are sufficiently large.
\begin{thm}
\label{thm:asymptotic}For any $k_{1},\dots,k_{p}\ge1$, there is
$n_{0}=n_{0}\left(k_{1},\dots,k_{p}\right)$ such that if $n_{1},\dots,n_{p}\ge n_{0}$
and if $\F\subseteq\kunif$ is
a non-trivially intersecting family, then
\begin{align*}
\left|\F\right| & \le M^{\max}\left(n_{1},\dots,n_{p},k_{1},\dots,k_{p}\right):= \max M_{t,S}^{\HM}
\end{align*}
where the maximum is over all $t\in \range p$ and $S\subseteq [p]\setminus \{t\}$, except the case $S=\emptyset$ if $k_t=1$, and the case $S=\{r\}$ if $k_t=k_r=1$, for some $r$.
\end{thm}

We remark that if $p\le 2$ and $k_s=1$ for all $s$, then there is no non-trivially intersecting family, in which case \cref{thm:asymptotic} holds vacuously. We also remark that none of the values $M_{t,S}^{\HM}$ in \cref{thm:asymptotic} are redundant. For example, for any $p$, any $t\in [p]$ and any $S\subseteq [p]\setminus \{t\}$, one can show (via a rather tedious computation) that $M_{t,S}^{\HM}$ is the unique maximum in the expression in \cref{thm:asymptotic} in the following regime. Suppose that all the $k_s=2$, and $n_t$ is sufficiently large, and all the $n_s$, for $s\in S$, are equal to $n_t+1$, and all the $n_s$, for $s\notin S\cup \{t\}$, are equal and sufficiently large relative to $n_t$. Intuitively speaking, if the $n_s$ and $k_s$ are about the same size then we should take $S=[p]$; we should exclude parts from $S$ only when the parts are quite ``imbalanced''.

The proof of \cref{thm:asymptotic} is conceptually rather simple, though the details are nontrivial. We first define a notion of ``shiftedness'' and in \cref{sec:shifting} we show that there is a non-trivially intersecting family of maximum size which is shifted. So, it will suffice to prove \cref{thm:asymptotic} for shifted families, which we do in \cref{sec:proof}. Our notion of shiftedness forces enough structure that it is possible to prove that a maximum-size non-intersecting family must be of a certain parameterized form. We can explicitly write the size of the family as a function of the parameters, and then it remains to optimize this expression over choices of the parameters.

Finally, we remark that the case where $n_1=\dots=n_p=n$ and $k_1=\dots=k_p=1$ is of special interest. In this case, $\kunif$ corresponds to $K^p_n$ (the $p$th tensor power of an $n$-vertex clique), and an intersecting family $\F\subseteq\kunif$ corresponds to an independent set in $K^p_n$. The problem of characterizing maximum independent sets in $K^p_n$ was solved by Greenwell and Lov\'asz~\cite{GL74} well before Frankl considered the general multi-part Erd\H os-Ko-Rado problem, and there has actually already been interest in proving stability theorems in this setting. Indeed, a result in the above-mentioned paper by Alon, Dinur, Friedgut and Sudakov~\cite{ADFS04}, improved by Ghandehari and Hatami~\cite{GH08}, says that if an independent set in $K^p_n$ has almost maximum size, then it is ``close'' in structure to a maximum family. When $n$ is large relative to $p$, then \cref{thm:asymptotic} gives a much stronger stability result. Actually, in this simplified setting one can use the machinery in \cref{sec:shifting} to give a simple proof of a non-asymptotic version of \cref{thm:asymptotic}.

\begin{thm}
\label{thm:k1}
For any $n\ge 2$ and $p\ge 3$, suppose $\F\subseteq \binom{[n]}{1}^{p}$ is a non-trivially intersecting family. Then
\[
|\F|\le \F^\HM_{1,\range p\setminus \{1\}}(n,\dots,n,1,\dots,1)=n^{p-1}-(n-1)^{p-1}+n-1.
\]
\end{thm}

A proof of \cref{thm:k1} appeared in a previous version of this paper (which can still be found on the arXiv), but we were since informed that \cref{thm:k1} is actually a special case of a theorem proved a few years ago by Borg~\cite{Bor13a} concerning intersecting families of ``signed sets''. His proof follows basically the same approach.

\section{Shifting for multi-part families}
\label{sec:shifting}

First we define our notion of shiftedness via a shifting operation that makes a family
more structured without interfering too much with its size or intersection
properties. Our shifting operation will be a multi-part adaptation of the well-known shifting operation for single-part families, introduced by Erd\H os, Ko and Rado~\cite{EKR61} (see also~\cite{Fra87} for a survey).
\begin{defn}
[Shifting]For $t\in [p]$, $1\le i<j\le n_{t}$, and $F=\bigsqcup_{s}F_{s}\in \kunif$,
define $S_{t}^{i,j}\left(F\right)$ as follows. If $i\in F_{t}$ or
$j\notin F_{t}$ then $S_{t}^{i,j}\left(F\right)=F$. Otherwise, $S_{t}^{i,j}\left(F\right)$
is defined by replacing $F_{t}$ with $\left(F_{t}\backslash\left\{ j\right\} \right)\cup\left\{ i\right\} $.
For a family $\F\subseteq\kunif$, define
\[
S_{t}^{i,j}\left(\F\right)=\left\{ S_{t}^{i,j}\left(F\right):F\in\F\right\} \cup\left\{ F:F,S_{t}^{i,j}\left(F\right)\in\F\right\} .
\]
That is, change $\F$ by shifting every $F\in\F$ except those that
would cause a conflict. We call the result a \emph{$t$-shift} (or simply a \emph{shift}) of $\F$. A family $\F\subseteq\kunif$ is \emph{$t$-shifted} if it is stable under $t$-shifts, and a family is \emph{shifted} if it is $t$-shifted for all $t\in \range p$.
\end{defn}
In the single-part case, it is well-known (see for example~\cite[Propositions~2.1-2]{Fra87}) that any shift of any intersecting family is again an intersecting family of the same size, and one can always obtain a shifted family by repeatedly applying shifting operations. The corresponding facts for the multi-part case immediately follow. In particular, it follows that we can always reach a shifted family by repeated shifts.



It is not in general true that a shift of a non-trivially intersecting family is again non-trivially intersecting, but in the single-part case, there is nevertheless a non-trivially intersecting family of maximum size which is shifted. This was recently proved by Kupavskii and Zakharov~\cite{KZ16}, and was also previously proved in a more general setting by Borg~\cite{Bor13b}. In the multi-part case this fact is still true, but is more difficult to prove, as follows.

\begin{thm}
\label{thm:shifted-ok}
For any $p\ge 1$, any $k_1,\dots,k_p$ and any $n_1,\dots,n_p$, if there is a non-trivially intersecting family $\F\subseteq\kunif$ then there is a maximum-size non-trivially intersecting family $\F'\subseteq\kunif$ which is shifted.
\end{thm}

\cref{thm:shifted-ok} will follow immediately from a sequence of three lemmas. Say a non-trivially intersecting family $\F\subseteq\kunif$ is \emph{$Q$-shifted} if it is $s$-shifted for each $s\in Q$, and if for each $s\notin Q$ there are $i_s<j_s$ such that $S^{i_s,j_s}_s(\F)$ is trivially intersecting. In particular, note that if $Q=[p]$ and $\F$ is $Q$-shifted then $\F$ is shifted. We now prove the following very basic lemma concerning $Q$-shiftedness.

\begin{lem}
\label{lem:shifted-1}
Under the assumption of \cref{thm:shifted-ok}, there is a maximum-size non-trivially intersecting family $\F'\subseteq\kunif$ which is $Q$-shifted for some $Q\subseteq [p]$.
\end{lem}

\begin{proof}
Let $\F\subseteq \kunif$ be a maximum-size non-trivially intersecting family. Starting with $\F_1=\F$, for $m\ge 1$ as long as there is a shift $S_{s}^{i,j}$ for which the family $S_{s}^{i,j}(\F_{m})$ is non-trivially intersecting and different from $\F_{m}$, set $\F_{m+1}=S_{s}^{i,j}(\F_{m})$ (for any such choice of $s,i,j$). Similarly to the single-part setting, this process must eventually terminate with a non-trivially intersecting family $\F'$ of size $|\F'|=|\F|$ such that for any $s\in [p]$ either $\F'$ is $s$-shifted or there are $1\le i_s < j_s\le n_s$ such that $S_{s}^{i_s,j_s}(\F')$ is trivially intersecting. With $Q=\{s\in [p]:\F' \text{ is } s\text{-shifted}\}$, it follows that $\F'$ is $Q$-shifted.
\end{proof}

\begin{lem}
\label{lem:Q-shifted1}
Suppose there is a maximum-size non-trivially intersecting family $\F\subseteq \kunif$ which is $Q$-shifted for some $Q\subsetneq [p]$ such that $k_t>1$ for some $t\notin Q$. Then there is a maximum-size non-trivially intersecting family $\F'$ which is shifted.
\end{lem}

\begin{proof}
Here we adapt the approach of Kupavskii and Zakharov~\cite{KZ16}. Since $\F$ is $Q$-shifted, for each $s\notin Q$ there are $i_s<j_s$ such that $S^{i_s,j_s}_s(\F)$ is trivially intersecting. By permuting the elements of the ground set in each part $s\notin Q$ in such a way that $i_s\mapsto 1$ and $j_s\mapsto 2$, we may assume that $(i_s,j_s)=(1,2)$ for each $s\notin Q$.

Let $\mathcal A$ denote the (non-empty) family of all sets in $\kunif$ which contain both $1$ and $2$ in part $t$. Since $\F$ is non-trivially intersecting but $S_{t}^{1,2}(\F)$ is trivially intersecting it follows that every set in $\F$ contains $1$ or $2$ in part $t$. Therefore, by maximality of $\F$, we must have $\mathcal A\subseteq \F$. Observe that $\mathcal A$ is a shifted family and that the only elements that belong to every set in $\mathcal A$ are $1$ and $2$ in part $t$. Since $\F$ is non-trivially intersecting and contains $\mathcal A$, repeatedly applying to $\F$ shifts of the form $S_{t}^{i,j}$ with $3\le i<j\le n_{t}$, and $s$-shifts with $s\ne t$, cannot create a trivial intersection. Call such shifts \emph{safe-shifts}. Starting from $\F$ repeatedly perform safe-shifts to obtain a non-trivially intersecting family $\F''$ that contains $\mathcal A$ and is stable under such shifts.

Now, by non-triviality of $\F''$, there is a set $F\in \F''$ containing $1$ but not $2$ in part $t$. Note that $F':=\range{k_{1}}\sqcup\dots\sqcup\left\{ 1,3,4,\dots,k_{t}+1\right\} \sqcup\dots\sqcup\range{k_{p}}$
can be reached from $F$ by a sequence of safe-shifts, and therefore $F'\in \F''$. Similarly, there is a set $G\in \F''$ containing $2$ but not $1$ in part $t$ and one can reach the set $G':=\range{k_{1}}\sqcup\dots\sqcup\left\{ 2,3,4,\dots,k_{t}+1\right\} \sqcup\dots\sqcup\range{k_{p}}$ from $G$ by a sequence of safe-shifts, implying that $G'\in \F''$.

Set $\G :=\bigcup_{s}\G_{s}$ where $\G_{s}:=\left\{ \range{k_{1}}\right\} \times\dots\times\binom{\range{k_{s}+1}}{k_{s}}\times\dots\times\left\{ \range{k_{p}}\right\}$. Note that the only sets in $\G$ which are not in $\mathcal A$ are precisely $F'$ and $G'$. Therefore, we conclude that $\G\subseteq \F''$. Since $\G$ is a shifted non-trivially intersecting family, further shifts applied to $\F''$ cannot create a trivial intersection. Therefore, we can repeatedly apply shifts to $\F''$ in order to obtain a shifted maximum-size non-trivially intersecting family $\F'$.
\end{proof}

\begin{lem}
Suppose there is a maximum-size non-trivially intersecting family $\F\subseteq \kunif$ which is $Q$-shifted for some $Q\subsetneq [p]$ such that $k_t=1$ for every $t\notin Q$. Then there is a maximum-size non-trivially intersecting family $\F'$ which is shifted.
\end{lem}

\begin{proof}
We remark that this is the part of the proof with the bulk of the new ideas specific to the multi-part case.
Define the \emph{order} of a family $\F\subseteq\kunif$ as follows:
\[\text{ord}(\F):=\sum_{\bigsqcup_{s}\! F_{s}\in \F}\sum_{s\in [p]}\sum_{x\in F_s}x.\]
Consider $\F$ as in the lemma statement, chosen to have minimum possible order, and suppose without loss of generality that $\F$ is $Q$-shifted with $Q=\{q+1,q+2,\dots, p\}$ for some $q\ge 1$.

Since $\F$ is $Q$-shifted, for each $t\in \range q$ there are $i_t<j_t$ such that $S^{i_t,j_t}_t(\F)$ is trivially intersecting. This implies that every set in $\F$ contains $i_t$ or $j_t$ in each part $t\in \range q$. Since $k_t=1$ for each $t\in\range q$, this actually implies that each set is \emph{equal} to $\{i_t\}$ or $\{j_t\}$ in each such part. Note that actually $(i_t,j_t)=(1,2)$ for each $t\in\range q$, because otherwise we could permute the elements of the ground set in each part $t\in\range q$ in such a way that $i_t\mapsto 1$ and $j_t\mapsto 2$, and obtain a maximum-size $Q$-shifted non-trivially intersecting family of smaller order. Therefore, for any $\bigsqcup_{s}F_{s}\in \F$ and any $t\in [q]$, we have $F_{t}= \{1\}$ or $F_{t}=\{2\}$. Furthermore, since $\mathcal F$ is $t$-shifted for $t\in \{q+1,q+2,\dots, p\}$, we know that if $\bigsqcup_{s}F_{s}\in \mathcal F$ then also $F_{1}\sqcup\dots\sqcup F_{q}\sqcup\range{k_{q+1}} \sqcup\dots\sqcup\range{k_{p}}\in \mathcal F$. Consider the set
\[\mathcal H = \{(x_1,\dots, x_{q})\in \{1,2\}^{q}: \{x_1\}\sqcup\dots\sqcup \{x_{q}\}\sqcup\range{k_{q+1}} \sqcup\dots\sqcup\range{k_{p}}\in \mathcal F\},\]
which we can think of as the set of possible projections onto the first $q$ coordinates of sets in $\mathcal F$. We view $\mathcal H$ as a subset of the hypercube $\{1,2\}^q$.
For any $P= (x_1,\dots, x_q)\in \{1,2\}^{q}$ define its \emph{complement} $\overline{P} := (3-x_1,\dots, 3-x_q)\in \{1,2\}^{q}$. Note that if $\overline{P}\notin \mathcal H$ then $P$ shares at least one coordinate with every point in $\mathcal H$. This implies that $\{x_1\}\sqcup\dots\sqcup \{x_{q}\}\sqcup\range{k_{q+1}} \sqcup\dots\sqcup\range{k_{p}}$ intersects any set in $\mathcal F$. By maximality of $\mathcal F$, we conclude that $P\in \mathcal H$. Therefore, for any pair $\{P, \overline{P}\}$, with $P\in \{1,2\}^{q}$, at least one of $P$ or $\overline{P}$ is in $\mathcal H$. In particular, since these pairs partition $\{1,2\}^{q}$ into $2^{q-1}$ parts, we conclude that $|\mathcal H|\ge 2^{q-1}$.

Next, suppose that $\mathcal H$ contains two Hamming-adjacent points. By the definition of $\mathcal H$, this would imply that for some $s\in [q]$ and $x_1,\dots, x_{s-1},x_{s+1},\dots, x_q\in\{1,2\}$, for both choices of $y\in \{1,2\}$, we have
\[\{x_1\}\sqcup\dots\sqcup \{x_{s-1}\}\sqcup\{y\}\sqcup \{x_{s+1}\}\sqcup\dots\sqcup \{x_{q}\}\sqcup\range{k_{q+1}} \sqcup\dots\sqcup\range{k_{p}}\in \F.\]
This would mean that $S_{s}^{1,2}(\mathcal F)$ is non-trivially intersecting, contradicting the fact that $s\notin Q$. Hence, any two points in $\mathcal H$ are at Hamming distance at least $2$. For any $s\in [q]$ we can partition $\{1,2\}^{q}$ into $2^{q-1}$ pairs of points which differ only in coordinate $s$, so we conclude that $|\mathcal H|\le 2^{q-1}$.

By the last two paragraphs, we conclude that $|\mathcal H| = 2^{q-1}$ and
\begin{enumerate}
\item[(1)] $\mathcal H$ contains exactly one point from any pair of adjacent points in $\{1,2\}^{q}$,
\item[(2)] $\mathcal H$ contains exactly one point from any pair of complementary points in $\{1,2\}^{q}$.
\end{enumerate}
Note that (1) implies that $\mathcal H$ and $\{1,2\}^{q}\setminus \mathcal H$ form a partition of the hypercube $\{1,2\}^{q}$ into independent sets. Since the hypercube $\{1,2\}^{q}$ is a connected bipartite graph, this partition is unique up to switching the parts. Therefore, we conclude that either $\mathcal H = \mathcal H_{\text{even}} := \{(x_1, \dots, x_q)\in \{1,2\}^{q}: x_1+\dots + x_q\text{ is even}\}$ or $\mathcal H = \mathcal H_{\text{odd}} := \{(x_1, \dots, x_q)\in \{1,2\}^{q}: x_1+\dots + x_q\text{ is odd}\}$. Note that if $q=2$ (or more generally, if $q$ is even) then neither $\mathcal H_{\text{even}}$ nor $\mathcal H_{\text{odd}}$ satisfy (2). Moreover, $q\neq 1$ as otherwise $\mathcal F$ would be trivially intersecting. It follows that $q\ge 3$.

Choose $P=(x_1,\dots, x_q)\in \mathcal H$ such that $x_1+\dots +x_q$ is as large as possible (that is, $P=(2,2,\dots, 2)$ if $\mathcal H =\mathcal H_{\text{even}}$ and say $P=(2,2,\dots, 2,1)$ if $\mathcal H= \mathcal H_{\text{odd}}$). Let $\mathcal F''$ denote the family obtained from $\mathcal F$ by replacing every set of $\mathcal F$ of the form $\{x_1\}\sqcup\dots\sqcup \{x_{q}\}\sqcup F_{q+1}\sqcup\dots\sqcup F_{p}$ with $\{3-x_1\}\sqcup\dots\sqcup \{3-x_{q}\}\sqcup F_{q+1}\sqcup\dots\sqcup F_{p}$. Note that the latter set is not in $\F$ since $\overline{P}\notin \mathcal H$ by (2). Therefore, $|\mathcal F''| = |\mathcal F|$. Moreover, $\text{ord}(\mathcal F'')< \text{ord}(\mathcal F)$ by the choice of $P$ and the fact that $q\ge 3$. We claim now that $\mathcal F''$ is a non-trivially intersecting family, and is still $s$-shifted for $s\in Q$. Indeed, since we removed from $\mathcal F$ every set whose first $q$ parts agree with $P$, every set in $\mathcal F''$ has at least one of its first $q$ parts agreeing with $\overline{P}$. Hence, every set in $\mathcal F''\setminus \mathcal F$ intersects any set in $\mathcal F''$ in one of the first $q$ parts, and so $\mathcal F''$ is an intersecting family. It remains to show that $\F''$ is non-trivially intersecting. If this were not the case, then for some part $s\in [q]$ all the sets in $\F''$ would have to agree with $\overline{P}$ in part $s$. However, since $q\ge 3$ there is some point $P'\in \mathcal H$ at Hamming distance $2$ from $P$ which agrees with $P$ in coordinate $s$. Since $P'\in \mathcal H$, there is some set in $\mathcal F$ which coincides with $P'$ in its first $q$ parts and therefore belongs to $\F''$ and does not agree with $\overline{P}$ in part $s$. Therefore, $\mathcal F''$ is non-trivially intersecting, as claimed. Note that $\F''$ is still $s$-shifted for $s\in Q$, and note that shifting $\F''$ can only decrease its order further. Therefore, by repeatedly applying shifts to $\F''$ which do not make the family trivially intersecting (as in \cref{lem:shifted-1}), we can obtain a maximum-size non-trivially intersecting family $\F'$ which is $Q'$-shifted for some $Q'\subseteq [p]$ and which has smaller order than $\F$. By the order-minimality of the choice of $\F$ we conclude that either $Q'=[p]$ or $Q'\subsetneq [p]$ and $k_t>1$ for some $t\notin Q'$. In the first case, it follows that $\F'$ is shifted, as desired. In the second case, the result follows from \cref{lem:Q-shifted1}.
\end{proof}


\section{Proof of \texorpdfstring{}{Theorem }\cref{thm:asymptotic}}
\label{sec:proof}

\global\long\def\Q{P}
\global\long\def\cQ{\mathcal{\Q}}

In this section we prove \cref{thm:asymptotic}. By \cref{thm:shifted-ok}, it suffices to prove the required bound for shifted non-trivially intersecting families $\mathcal \F\subseteq\kunif$.

The crucial property of shifted families is that one can observe that they are intersecting just by looking at the first few elements of each part. For $F=\bigsqcup_s F_s\in\bigsqcup_s 2^{\range {n_s}}$, let $\Q_{s}\left(F\right)=F_s\cap\range{2k_{s}}$
be the ``projection'' onto the first $2k_{s}$ elements of part $s$, let $\Q\left(F\right)=\bigsqcup_{s}\Q_{s}\left(F\right)$, and let 
$\cQ(\F) =\{\Q(F):F\in \F\}$.

\begin{lem}
\label{lem:alpha-shifted}
For any $p\ge 1$, any $k_1,\dots,k_p$ and any $n_1,\dots,n_p$, if $\F\subseteq\kunif$ is intersecting and shifted then $\Q(F)\cap \Q(G)\ne\emptyset$ for any $F,G\in \F$.
\end{lem}
\begin{proof}
Suppose that there were $F=\bigsqcup_s F_s,\;G=\bigsqcup_s G_s\in\F$ such that $\Q\left(F\right)$ and
$\Q\left(G\right)$ are disjoint, and let $B_{s}=F_s\cap G_s$. Note that $2\left|B_{s}\right|\le 2 k_{s}-|\Q_{s}\left(F\right)|-|\Q_{s}\left(G\right)|$ since $B_{s}\subseteq F_s\setminus \Q_{s}(F)$ and $B_{s}\subseteq G_s\setminus \Q_{s}(G)$.
For each $s$, choose some $B_{s}'\subseteq\range{2k_{s}}\backslash(\Q_{s}\left(F\right)\cup \Q_{s}\left(G\right))$
with $\left|B_{s}'\right|=\left|B_{s}\right|$ and note that $F':=\bigsqcup_{s}\left(F_s\backslash B_{s}\right)\cup B_{s}'$ can be obtained from $F$ by a sequence of shifts. Therefore, by shiftedness $F'\in\F$. But $G$ is disjoint from $F'$, which
is a contradiction.
\end{proof}

Let $\F$ be a shifted non-trivially intersecting family of maximum size. \cref{lem:alpha-shifted} implies that $|\Q(F)|\ge 2$ for each $F\in \F$, because otherwise $\cQ(\F)$, and therefore $\F$, would be trivially intersecting. Let $\F^*=\{F\in \F:|\Q(F)|=2\}$. Now, since $\cQ\left(\F^{*}\right)$ is a 2-uniform intersecting family, it has very restricted structure. We distinguish between two cases, depending on whether $\cQ\left(\F^{*}\right)$ is trivially intersecting.

All asymptotic notation in the following sections is to be taken as $n_0\to\infty$, treating $p,k_1,\dots,k_p$ as constants. In particular, this means that there are only $O(1)$ possibilities for a projected family $\cQ(\F)$. Where relevant we assume that $n_0$ is sufficiently large.

\subsection{Case 1: $\cQ\left(\F^{*}\right)$ is trivially intersecting (or empty)}

For this case, without loss of generality assume $n_1\le \dots\le n_p$. It will be important to note that for any $Z\in \bigsqcup_s2^{\range{n_s}}$ with $z_s\le k_s$ elements in each part $s$, there are $\Theta\left(\prod_{s}n_s^{k_{s}-z_s}\right)$ sets in $\kunif$ which include $Z$. In particular, note that
\[
M^{\max}=\begin{cases}
\Theta\left(\left(\prod_{s}n_s^{k_{s}}\right)/\left(n_{1}n_{2}\right)\right) & \text{if $k_{1}=1$;}\\
\Theta\left(\left(\prod_{s}n_s^{k_{s}}\right)/n_{1}^{2}\right) & \text{otherwise},
\end{cases}\label{eq:Mmax}
\]
and note that if $|Z|>2$ then there are at most $o(M^{\max})$ sets in $\kunif$ which contain $Z$. This means that $\left|\F\right|=\left|\F^{*}\right|+o\left(M^{\max}\right)$. It follows that if $|\F^*|$ is empty, then $|\F|<M^{\max}$. So, in what follows assume $\F^*$ is not empty, in which case $\cQ\left(\F^{*}\right)$ is non-empty.

Assume that $\cQ\left(\F^{*}\right)=\{\{x,y_1\},\{x,y_2\},\dots\{x,y_q\}\}$ with $q\ge 1$, and let $t$ be the part of $x$. Noting that $q\le \sum_s 2k_s=O(1)$, we have
\[
|\F^*|\le \sum_{i=1}^q \left|\left\{F\in \kunif:\{x,y_i\}\subseteq F\right\}\right|=O\left(\left(\prod_{s}n_{s}^{k_{s}}\right)/(n_t n_1)\right).
\]
This means that in order to have $|\F^*|=\Omega(M^{\max})$ (which is necessary to have $|\F|>M^{\max}$), we must have $n_t=O(n_1)$ in the case $k_1>1$, and $n_t=O(n_2)$ in the case $k_1=1$.

Now, we prove that every set in $\F$ is of a very specific form. Let $Y=\{y_1,\dots,y_q\}$.

\begin{claim}
Every set $F\in\F$ has $x\in F$ and $F\cap Y\ne\emptyset$
or satisfies $F\supseteq Y$.
\end{claim}

\begin{proof}
Note that every $F\in\F$ with $x\notin\Q\left(F\right)$ must
have $Y\subseteq\Q\left(F\right)$ in order for $\Q\left(F\right)$ to intersect everything in $\cQ\left(\F^{*}\right)$. Moreover, there must be at least one such $F$ because $\F$ is non-trivially intersecting.

We now claim that every $F\in \F$ intersects $Y$. Let $\F^Y\subset \F$ be the non-empty subfamily of sets $F\supseteq Y$ not containing $x$, and assume for the purpose of contradiction that there is a set in $\F$ which does not intersect $Y$. This means that $\cQ(\F^Y)$ does not contain $Y$, because $\cQ(\F)$ is intersecting. Choose $z\notin Y$ in some set in $\cQ(\F^Y)$, in a part $r$ for which $n_r$ is as small as possible, and set $Z=Y\cup \{z\}$. Now consider the family $\F'$ consisting of all the sets in $\kunif$ which include $Z$, in addition to all the sets in $\kunif$ that contain $x$ and intersect $Z$. Note that this family is non-trivially intersecting. We will show that $|\F'|>|\F|$, contradicting maximality.

Observe that $\cQ(\F')$ includes $\cQ(\F^*)=\{\{x,y_1\},\dots,\{x,y_q\}\}$, so for each $F\in \F\setminus\F'$, the projection $\Q(F)$ has size at least 3. Moreover we claim that each such $\Q(F)$ contains some $w\notin Y$ appearing in some set in $\cQ(\F^Y)$. Indeed, this holds if $F\in \F^Y$ since $Y$ is not in $\cQ(\F^Y)$. Otherwise, if $F\in \F\setminus(\F'\cup\F^Y)$ then $F\cap Y=\emptyset$ yet $\Q(F)$ intersects every set in $\cQ(\F^Y)$. Note that the part of $w$ has size at least $n_{r}$ by the choice of $z$. Noting that if $k_1=1$ then $\Q(F)$ can contain at most one element from part 1, and recalling that there are only $O(1)$ different possibilities for $\Q(F)$ among $F\in \F\setminus\F'$, we have
\[
\left|\F\setminus\F'\right|=
\begin{cases}
O\left(\left(\prod_{s}n_{s}^{k_{s}}\right)/\left(n_1^2n_r\right)\right)&\text{if $k_1>1$;}\\
O\left(\left(\prod_{s}n_{s}^{k_{s}}\right)/\left(n_1n_2n_r\right)\right)&\text{otherwise.}
\end{cases}
\]
Recall that if $k_1>1$ then $n_t=O(n_1)$, and if $k_1=1$ then $n_t=O(n_2)$. It follows that $\left|\F\setminus\F'\right|=o\left(\left(\prod_{s}n_{s}^{k_{s}}\right)/\left(n_{t}n_{r}\right)\right)$. On the other hand,
\[
|\F'\setminus \F|\ge \left|\left\{F\in\kunif:\Q(F)=\{x,z\}\right\}\right|=\Theta\left(\left(\prod_{s}n_{s}^{k_{s}}\right)/\left(n_{t}n_{r}\right)\right).
\]
So, $\F$ was not of maximum size, which is a contradiction.
\end{proof}
We have proved that every set $F\in\F$ has $x\in F$ and $F\cap Y\ne\emptyset$, or $Y\subseteq F$. Note that we must have $|Y|\ge2$ since $\F$ is non-trivially intersecting. By maximality, $\F$ must actually be the family of all $F\in \kunif$ satisfying $x\in F$ and $F\cap Y\ne\emptyset$, or $Y\subseteq F$, since this family is still non-trivially intersecting. If we had $|Y|=2$ then we would have $Y\in \cQ(\F^*)$, contradicting the fact that $\cQ(\F^*)=\{\{x,y_1\},\dots,\{x,y_q\}\}$. So, $|Y|>2$.

Note that $|\F|$ is entirely determined by $t$ and each $|\Q_s(Y)|$. Note that if $k_t=1$ then $|\Q_t(Y)|=0$, since $\{x,y\}\in \cQ(\F^*)$ for each $y\in Y$. Alternatively, if $k_t>1$, then we claim $|\Q_t(Y)|=k_t$. Suppose not, and let $z\notin Y\cup\{x\}$ be an element of part $t$. Let $\F'$ be the family of all $F\in \kunif$ such that $x\in F$ and $F\cap (Y\cup \{z\})\ne\emptyset$, or $Y\cup \{z\}\subseteq F$. We remark that this is a non-trivially intersecting family. Note that, since $|Y|>2$,
\[
|\F\setminus\F'|\le \left|\left\{F\in\kunif:Y\subseteq F\right\}\right|=
\begin{cases}
O\left(\left(\prod_{s}n_s^{k_{s}}\right)/n_1^3\right)& \text{if $k_1>1$,}\\
O\left(\left(\prod_{s}n_s^{k_{s}}\right)/(n_1n_2^2)\right)&\text{if $k_1=1$.}
\end{cases}
\]
On the other hand,
\[
|\F'\setminus\F|\ge\left|\left\{F\in\kunif:x\in F,\,F\cap (Y\cup \{z\})=\{z\}\right\}\right|=\Omega\left(\left(\prod_{s}n_s^{k_{s}}\right)/n_t^2\right).
\]
Recalling that $n_t=O(n_1)$ if $k_1>1$ and $n_t=O(n_2)$ otherwise, this contradicts the maximality of $\F$. We have proved that $|\Q_t(Y)|=k_t$.

Now, let $L_t$ be the set of sequences $\ll\in \prod_{s=1}^p\{0,\dots,k_s\}$ satisfying the following conditions:
\begin{itemize}
\item $\sum_s \ell_s\ge2$,
\item $\ell_t=0$ if $k_t=1$,
\item $\ell_t=k_t$ if $\sum_s \ell_s>2$ and $k_t>1$.
\end{itemize}
For each $\ll\in L_t$, define
\[
Y_\ll=\range{\ell_1}\sqcup\dots\sqcup\range{\ell_{t-1}}\sqcup\{2,\dots,\ell_t+1\}\sqcup\range{\ell_{t+1}}\sqcup\dots\sqcup\range{\ell_p},
\]
and let $\F_{t,\ll}$ be the family of all $F=\bigsqcup_s F_s\in \kunif$ such that $1\in F_t$ and $F\cap Y_\ll\ne \emptyset$, or $Y_\ll\subseteq F$. We remark that this family is shifted and non-trivially intersecting. Note that
\begin{align*}
|\F_{t,\ll}|=:M_{t,\ll}&=\binom{n_t-\ell_t-1}{k_t-\ell_t}\prod_{s\ne t} \binom{n_s-\ell_s}{k_s-\ell_s}+\binom{n_t-1}{k_t-1}\prod_{s\ne t} \binom{n_s}{k_s}-\binom{n_t-\ell_t-1}{k_t-1}\prod_{s\ne t} \binom{n_s-\ell_s}{k_s}.
\end{align*}
The significance of the families $\F_{t,\ll}$ is that $|\F|=M_{t,\ll}$ for $\ll\in L_t$ defined by $\ell_s=|\Q_s(Y)|$.

Observe that for any $S\subseteq\range p\setminus \{t\}$ as in the statement of \cref{thm:asymptotic}, we have $\F^\HM_{t,S}=\F_{t,\ll(S)}$, where $\ll(S)\in L_t$ is defined by
\[\ell_s(S)=
\begin{cases}
k_s&\text{if $k_t>1$ and $s=t$, or if $s\in S$;}\\
0&\text{otherwise.}
\end{cases}\]
In the remainder of this section we show that $M_{t,\ll}$ (as a function of $\ll\in L_t$) is maximized when $\ll=\ll(S)$ for some $S\subseteq\range p\setminus \{t\}$, or when $\sum_s\ell_s=2$. For any $s\in\range p\setminus \{t\}$ and any $\ell_1,\dots,\ell_{s-1},\ell_{s+1},\dots,\ell_p$, let $\ll(x)=(\ell_1,\dots,\ell_{s-1},x,\ell_{s+1},\dots,\ell_p)$ and define $f:x\mapsto M_{t,\ll(x)}$.
\begin{claim}
\label{claim:unimodal}
The function $-f$ is unimodal (that is, monotone nondecreasing then monotone nonincreasing).
\end{claim}
\begin{proof}
Note that $f(x)$ is of the form $\alpha \binom{n_{s}-x}{a}-\beta\binom{n_{s}-x}{b}+c$ for $a\ge b\ge 1$ ($a= n_s-k_s$, $b=k_s$), $\alpha,\beta> 0$, and some $c\in \NN$. It suffices to prove that for all $b\ge 1$, $d\ge 0$ and $\gamma>0$, the function $g_{b,d,\gamma}:\NN\mapsto\mathbb R:y\mapsto \binom{y}{b}-\gamma\binom{y}{b+d}$ is unimodal. We can do this by induction on $b$, noting that the case $b=1$ follows from convexity of the function $y\mapsto\binom{y}{d+1}$. Let $\Delta$ be the difference operator, meaning that $\Delta h(y)=h(y+1)-h(y)$; we can compute $\Delta g_{b,d,\gamma}(y)=g_{b-1,d,\gamma}(y)$. Note that $g_{b-1,d,\gamma}(0)=0$ if $b>1$, so if $g_{b-1,d,\gamma}$ is unimodal then it is nonnegative until it reaches its maximum, after which point it is monotone nonincreasing. This means that if $g_{b-1,d,\gamma}$ ever becomes negative then it will never become positive after that point. That is to say, unimodality of $g_{b,d,\gamma}$ follows from unimodality of $g_{b-1,d,\gamma}=\Delta g_{b,d,\gamma}$.
\end{proof}

\cref{claim:unimodal} implies that $M_{t,\ll}$ is maximized on the boundary of $L_t$, implying that $\ell_s\in \{0,k_s\}$ for all $s\ne t$, or $\sum_s\ell_s=2$. Indeed, otherwise we could increase or decrease some $\ell_s$ without decreasing the value of $M_{t,\ll}$.
First consider the case where $\ell_s\in \{0,k_s\}$ for $s\ne t$, and $\sum_s\ell_s>2$. Since $\ll\in L_t$, we have $\ell_t=k_t$ if $k_t>1$, and $\ell_t=0$ if $k_t=1$. Therefore, $\ll=\ll(S)$ for $S$ being the set of all $s\ne t$ such that $\ell_s=k_s$. This means that $|\F|\le M_{t,\ll(S)}=M^\HM_{t,S} \le M^{\max}$. Next consider the case that $M_{t,\ll}$ is maximized for some $\ll=\ll^{\max}\in L_t$ satisfying $\sum_s \ell^{\max}_s=2$. Recall that $\F_{t,\ll^{\max}}$ is a shifted non-trivially intersecting family of maximum size. Moreover note crucially that in this case $\cQ(\F^*_{t,\ll^{\max}})$ is non-trivially intersecting. We will show in the next subsection that this implies $|\F_{t,\ll^{\max}}|\le M^{\max}$.

\subsection{Case 2: $\cQ\left(\F^{*}\right)$ is non-trivially intersecting}
\label{subsec:non-trivial}

If $\cQ(\F^*)$ is not trivially intersecting, it must consist of three sets of the form $\{x,y\},\{x,z\},\{y,z\}$. Therefore, by \cref{lem:alpha-shifted}, every set in $\F$ must contain at least two of $x,y,z$, and by maximality we can assume $\F$ is in fact the family consisting of every possible set in $\kunif$ with this property. Suppose without loss of generality that $x,y,z$ appear in parts $1,\dots,q$ (so $q\le 3$, with $q=3$ only when $x,y,z$ are in different parts). Let $\F'\subseteq\prod_{s\le q}\binom{\range{n_s}}{k_s}$ be the $q$-part family consisting of all sets containing at least two of $x,y,z$, so that
\[
|\F|=|\F'|\prod_{s>q}\binom{n_s}{k_s}.
\]
It suffices to prove that $|\F'|\le M^{\max}(n_1,\dots,n_q,k_1,\dots,k_q)$, because
\[
M^\HM_{t,S}(n_1,\dots,n_p,k_1,\dots,k_p)=M^\HM_{t,S}(n_1,\dots,n_q,k_1,\dots,k_q)\prod_{s>q}\binom{n_s}{k_s}
\]
for any $t\in [q]$ and $S\subseteq [q]\setminus\{t\}$.
In the case where $q=1$, this fact follows immediately from the Hilton-Milner theorem, because $\F'\subseteq\binom{\range{n_1}}{k_1}$ is a non-trivially intersecting single-part family and $M^{\max}(n_1,k_1)=M^\HM(n_1,k_1)$ is the Hilton-Milner bound.

The remaining cases $q=2$ and $q=3$ will require some rather tedious calculations. We write $f\sim g$ to denote $f=(1+o(1))g$ and we write $f\gtrsim g$ to denote $f\ge (1+o(1))g$. Recall that all asymptotics are taken as $n_0\to \infty$.

First, consider the case where $q=2$. Suppose without loss of generality that $x,y$ are in part 1 (and so $k_1\ge 2$) and $z$ is in part 2. Then, we compute
\begin{align*}
|\F'|&\sim \frac{n_{1}^{k_{1}}n_{2}^{k_{2}}}{k_1!k_2!}\left(\frac{k_1(k_1-1)}{n_1^2}+\frac{2k_1k_2}{n_1n_2}\right),\\
M^\HM_{1,\{2\}}(n_1,n_2,k_1,k_2)&\sim \frac{n_{1}^{k_{1}}n_{2}^{k_{2}}}{k_1!k_2!}\left(\frac{k_1^2(k_1-1)}{n_1^2}+\frac{k_1k_2^2}{n_1n_2}\right),\\
M^\HM_{2,\{1\}}(n_1,n_2,k_1,k_2)&\sim 
\begin{cases}
\frac{n_{1}^{k_{1}}n_{2}^{k_{2}}}{k_1!k_2!}\left(\frac{k_2^2(k_2-1)}{n_2^2}+\frac{k_1^2k_2}{n_1n_2}\right)&\text{if $k_2>1$}\\
\frac{n_{1}^{k_{1}}n_{2}^{k_{2}}}{k_1!k_2!}\cdot\frac{k_1^2k_2}{n_1n_2} + n_2&\text{if $k_2=1$}.
\end{cases}
\end{align*}
 
Suppose now that $k_2>1$. If $n_2$ is much smaller than $n_1$ (say $10k_1^{10}k_2^{10}n_2\le n_1$) then ${k_2^2(k_2-1)}/{n_2^2}$ is much larger than ${k_1(k_1-1)}/{n_1^2}+2{k_1k_2}/{(n_1n_2)}$ and therefore $M^\HM_{2,\{1\}}>|\F'|$. Otherwise, $n_1=O(n_2)$, so $1/n_1^2 =\Omega(1/(n_1n_2))$ and recalling that $k_1,k_2\ge 2$, we have
\[
M^\HM_{1,\{2\}}\gtrsim\frac{n_{1}^{k_{1}}n_{2}^{k_{2}}}{k_1!k_2!}\left(\frac{2k_1(k_1-1)}{n_1^2}+\frac{2k_1k_2}{n_1n_2}\right)=(1+\Omega(1))|\F'|>|\F'|.
\]
Alternatively, suppose that $k_2=1$. If $|\F'|>M^\HM_{1,\{2\}}$ then
\[\frac{k_1(k_1-1)}{n_1^2}+\frac{2k_1}{n_1n_2}\gtrsim \frac{k_1^2(k_1-1)}{n_1^2}+\frac{k_1}{n_1n_2},\]
implying that $1/n_2\gtrsim(k_1-1)^2/n_1$. However, if $|\F'|>M^\HM_{2,\{1\}}$ then
\[\frac{k_1(k_1-1)}{n_1^2}+\frac{2k_1}{n_1n_2}\gtrsim \frac{k_1^2}{n_1n_2},\]
implying that $(k_1-1)/n_1\gtrsim(k_1-2)/n_2$. So, $(k_1-1)/n_1\gtrsim (k_1-2)(k_1-1)^2/n_1$, which is a contradiction unless $k_1=2$. But in the case $k_1=2,\,k_2=1$, note that the definitions of $\F'$ and $\F^\HM_{2,\{1\}}$ are the same up to a permutation of the ground set and so $|\F'|=M^\HM_{2,\{1\}}$.

Next, consider the case where $q=3$. We compute
\begin{align*}
|\F'|&\sim \frac{n_{1}^{k_{1}}n_{2}^{k_{2}}n_{3}^{k_{3}}}{k_1!k_2!k_3!}\left(\frac{k_1k_2}{n_1n_2}+\frac{k_1k_3}{n_1n_3}+\frac{k_2k_3}{n_2n_3}\right),\\
M^\HM_{t,\range{3}\setminus\{t\}}(n_1,n_2,n_3,k_1,k_2,k_3)&\gtrsim \frac{n_{1}^{k_{1}}n_{2}^{k_{2}}n_{3}^{k_{3}}}{k_1!k_2!k_3!}
\cdot\frac{k_t}{n_t}\sum_{s\ne t} \frac{k_s^2}{n_s}
\end{align*}
for $t\in \{1,2,3\}$. We then consider the weighted average
\begin{align*}
M{^\text{avg}}:=\sum_{t\in \{1,2,3\}} \frac{\sum_{s\ne t} k_s}{2\sum_s k_s}M^\HM_{t,\range{3}\setminus\{t\}}
&\gtrsim \frac{n_{1}^{k_{1}}n_{2}^{k_{2}}n_{3}^{k_{3}}}{k_1!k_2!k_3!} \sum_{t\in \{1,2,3\}} \frac{k_t}{n_t}\sum_{s\ne t} \frac{k_s^2}{n_s}\\
&=\frac{n_{1}^{k_{1}}n_{2}^{k_{2}}n_{3}^{k_{3}}}{k_1!k_2!k_3!}\sum_{t\in \{1,2,3\}} \left(\frac{\sum_{s\ne t}k_s^2+k_t\sum_{s\ne t}k_s}{2\sum_s k_s}\prod_{s\ne t}\frac{k_s}{n_s}\right).
\end{align*}
(Note that the simplification for the second line is obtained by expanding the outer summation as well as the inner one).
Now, define the function $f:(\ZZ^+)^3\to \RR$ by $$(k_1,k_2,k_3)\mapsto {(k_1^2+k_2^2+k_1k_3+k_2k_3)}/{(2k_1+2k_2+2k_3)}.$$ One can show that $f(k_1,k_2,k_3)\ge 11/10=1+\Omega(1)$ unless two of $k_1,k_2,k_3$ are equal to 1. So, if at most one of $k_1,k_2,k_3$ is equal to 1 then $M{^\text{avg}}=(1+\Omega(1))|\F'|>|\F'|$, implying that at least one of the bounds $M^\HM_{t,\range{3}\setminus\{t\}}$ is bigger than $|\F'|$.

It remains to consider the case where at least two of $k_1,k_2,k_3$ are equal to 1. Without loss of generality, say $k_2=k_3=1$. Then, $|\F'|>M^\HM_{2,\{1,3\}}$ implies $1/(n_1n_3)\gtrsim (k_1-1)/(n_1n_2)$, while $|\F'|>M^\HM_{3,\{1,2\}}$ implies $1/(n_1n_2)\gtrsim (k_1-1)/(n_1n_3)$. These inequalities cannot simultaneously be satisfied unless $k_1\le2$. If $k_1=k_2=k_3=1$ then note that $\F'$ and $\F^\HM_{1,\{2,3\}}=\F^\HM_{2,\{1,3\}}=\F^\HM_{3,\{1,2\}}$ are the same up to a permutation of the ground set, so $|\F'|=M^\HM_{1,\{2,3\}}$. Finally, if $k_1=2$, $k_2=k_3=1$, assume without loss of generality that $n_2\le n_3$. Using the definitions of $|\F'|$ and $M^\HM_{2,\{1,3\}}$ we can directly compute
\begin{align*}
M^\HM_{2,\{1,3\}}-|\F'|
&=\left(\binom{n_1}{2}n_3-\binom{n_1-2}{2}(n_3-1)+(n_2-1)\right)\\
&\qquad-\left((n_1-1)(n_3-1)+(n_1-1)(n_2-1)+\binom{n_1}{2}\right)\\
&=(n_3-n_2)(n_1-2)\ge 0.
\end{align*}

\section{Concluding remarks}
In this paper we have investigated a natural multi-part generalization of the Hilton-Milner problem, solving this problem in the case where the parts are large. We found that, surprisingly, the extremal families need not be of ``product'' type.

There are a few immediate questions that remain open. First is the question of whether the bound in \cref{thm:asymptotic} remains valid when we do not make the assumption that the parts are large. As an intermediate problem, one might hope to prove a version of \cref{thm:asymptotic} where each $n_s$ is only assumed to be large relative to its corresponding $k_s$, not necessarily relative to all the $k_s$ at once.

Second, there is the question of characterizing the extremal families. It is certainly not true that the families $\F_{s,S}^\HM$ are the only extremal families up to isomorphism; under certain circumstances families of the type discussed in \cref{subsec:non-trivial} may also be of maximum size. We imagine that the proof of \cref{thm:asymptotic} can be adapted to characterize the extremal shifted families, when the part sizes are large, but it is less clear how to deal with potential non-shifted extremal families.

Finally, we think that the idea of adapting extremal theorems to a multi-part setting is interesting in general, and several natural problems remain unexplored. For example, we could ask for a multi-part generalization of Ahlswede and Khachatrian's celebrated complete intersection theorem~\cite{AK97}: what is the maximum size of an intersecting family $\F\subseteq\kunif$ such that every $F,F'\in \F$ intersect in at least $t$ elements? Ahlswede, Aydinian and Khachatrian~\cite{AAK98} studied and resolved a slightly different problem where the intersection sizes are restricted ``locally'' on a per-part basis; this gives rise to extremal families of product type, but we suspect the ``global'' problem might have a richer solution. We remark that in the setting of \cref{thm:k1} where $n_1,\dots,n_s=n$ and $k_1=\dots=k_p=1$, this problem is equivalent to the question of finding the largest possible set of length-$p$ strings over an alphabet of size $n$ with Hamming diameter at most $n-t$. This problem was solved independently by Ahlswede and Khachatrian~\cite{AK98} and Frankl and Tokushige~\cite{FT99}.

\textbf{Acknowledgement.} We would like to thank Peter Frankl for insightful discussions following an earlier version of this paper. In particular, he suggested that the notion of a \emph{kernel} (originally introduced in \cite{Fra78} with the name \emph{base}), could be applied to this problem, playing a similar role to our ``projected'' families $\cQ(\F)$ and giving an alternative proof of \cref{thm:asymptotic}. It would be interesting to see whether this alternative approach has advantages over ours, for example for removing or reducing the large-part-sizes constraint and for characterizing the extremal families.

\end{document}